\documentclass[11pt]{amsart}
\usepackage{amsmath,amssymb,amsthm,mathtools}
\usepackage[margin=1.05in]{geometry}
\usepackage{microtype}

\newtheorem{theorem}{Theorem}[section]
\newtheorem{lemma}[theorem]{Lemma}
\theoremstyle{remark}
\newtheorem{remark}[theorem]{Remark}
\newtheorem{proposition}[theorem]{Proposition}

\newcommand{\Ric}{\operatorname{Ric}}
\newcommand{\Sec}{\operatorname{Sec}}
\newcommand{\II}{\operatorname{II}}
\newcommand{\tr}{\operatorname{tr}}
\newcommand{\RR}{\mathbb{R}}
\newcommand{\cR}{\mathcal{R}}

\newcommand{\eps}{\varepsilon}

\title{An elementary counterexample to Escobar's Steklov conjecture on the three-ball}

\author{Alexandre Girouard}
\address{Département de mathématiques et de statistique,
Université Laval,
Québec, QC, Canada}
\email{alexandre.girouard@mat.ulaval.ca}

\author{Thomas Hélière}
\address{Département de mathématiques et de statistique,
Université Laval,
Québec, QC, Canada}
\email{thomas.heliere.1@ulaval.ca}


\begin{document}

\begin{abstract}
  We give an explicit polynomial deformation of the Euclidean metric on the unit three-ball which has positive Ricci curvature and strictly convex boundary but violates Escobar's conjectured lower bound for the first nonzero Steklov eigenvalue. Our proof is fully human-verifiable.
\end{abstract}

\maketitle

\section{Introduction}\label{sec:intro}
Let \((M^n,g)\) be a smooth compact Riemannian manifold with nonempty boundary.  The Steklov problem is
\[
 \Delta_g u=0\quad\text{in }M,
 \qquad \partial_\nu u=\sigma u\quad\text{on }\partial M,
\]
where \(\nu\) is the outward unit normal.  Its spectrum is nonnegative and discrete:
\(0=\sigma_0<\sigma_1\le\sigma_2\le\cdots\nearrow\infty\), where each eigenvalue is repeated according to its multiplicity.
See \cite{LevitinMangoubiPolterovich2023,GirouardPolterovich2017,ColboisGirouardGordonSher2024} for introductory and survey material on the spectral geometry of this problem.

In \cite{Escobar1997}, Escobar proved that in dimension at least three, nonnegative Ricci curvature and the lower boundary curvature bound \(\II\ge c g_{\partial M}>0\) imply \(\sigma_1>c/2\).  He subsequently conjectured the sharp inequality
\[
 \Ric_g\ge0,\qquad \II\ge c g_{\partial M}>0
 \quad\Longrightarrow\quad \sigma_1\ge c,
\]
with equality only for a Euclidean ball of radius \(1/c\) (see \cite[p.~115]{Escobar1999}).

Until the independently obtained counterexamples in \cite{SunWangWang2026}, Escobar's conjecture remained open under its original Ricci-only hypotheses (see Remark \ref{rem:SWW} below). Xia and Xiong~\cite{XiaXiong} proved the sharp conclusion under the stronger assumption $\Sec_g\ge0$, while Duncan and Kumar~\cite{DuncanKumar} sharpened Escobar's bound $\sigma_1>c/2$ without imposing a stronger curvature sign condition, obtaining quantitative lower bounds that also depend on geometric data in a collar of the boundary.

\subsection{Main result}

Let
\[
 B^3=\{(s,y_1,y_2)\in\RR\times\RR^2:s^2+y_1^2+y_2^2\le1\}
\]
with Euclidean metric $g_E$.
We use the outward unit normal $\nu$ and define the shape operator by
\[
S(X)=\nabla_X\nu,\qquad X\in T\partial B^3.
\]
Equivalently,
\[
\II(X,Y)=g(SX,Y),
\]
so that $S=g_{\partial B^3}^{-1}\II$.
In particular, the Euclidean unit sphere has shape operator $S=I$ and
\[
 \sigma_1(B^3,g_E)=1,
 \qquad
 \min_{\partial B^3}\lambda_{\min}(S)=1.
\]

\begin{theorem}\label{thm:main}
There is an explicit Cartesian polynomial tensor $\widehat k$ of degree six on a neighbourhood of $\overline{B^3}$ such that, for
$\widehat g_\varepsilon=g_E+\varepsilon\widehat k$,
every sufficiently small $\varepsilon>0$ satisfies
\[
 \Ric_{\widehat g_\varepsilon}>0,
 \qquad
 S_{\widehat g_\varepsilon}\ge c(\varepsilon)I>0,
 \qquad
 \sigma_1(B^3,\widehat g_\varepsilon)<c(\varepsilon),
\]
where
\[
 c(\varepsilon)=\min_{x\in S^2}\lambda_{\min}
 \bigl(S_{\widehat g_\varepsilon}(x)\bigr).
\]
In particular, Escobar's conjecture is false in dimension three.
\end{theorem}
The perturbation tensor $\widehat k$ is given explicitly in
Section~\ref{sec:explicit}, see equations~\eqref{eq:kansatz}--\eqref{eq:rdef}.
\begin{remark}\label{rem:SWW}
While this manuscript was being prepared for submission, we became aware of the recent preprint of Sun, Wang, and Wang~\cite{SunWangWang2026}, which disproves Escobar's conjecture in every dimension $n\ge 3$ by means of conformal deformations $g_t=e^{2t\Phi}g_E$
of the Euclidean ball.  At a structural level, the two constructions exploit the same first-order instability of the Euclidean equality case: positive linearized Ricci curvature is combined with a favorable first variation of the boundary convexity and with the coordinate function as a Steklov Rayleigh competitor.  In their conformal setting these requirements reduce to three conditions on the scalar factor $\Phi$, namely positivity of the linearized Ricci tensor, increase of the boundary principal curvatures, and decrease of the coordinate Rayleigh quotient.  Up to a constant conformal factor, equivalently a homothetic normalization, this may be viewed in the same form used here: the coordinate Rayleigh quotient is stationary to first order while the whole boundary shape operator increases strictly.

The actual metric deformations, however, are quite different.  Their examples are conformal to the Euclidean metric, whereas ours is an additive,  nonconformal polynomial perturbation
$\widehat{g}_\varepsilon=g_E+\varepsilon\widehat{k}.$
In particular, the present three-dimensional construction is designed so that all of the essential Ricci, boundary-convexity, and Rayleigh-quotient estimates admit direct exact verification by hand.
\end{remark}

\subsection{A free-boundary consequence}\label{sec:fraser-li}
Fraser and Li~\cite{FraserLi2014} proved that if $M^n$ is a compact orientable Riemannian
manifold with
\[
\Ric_M\ge0,
\qquad
\II_{\partial M}\ge k\,g_{\partial M}>0,
\]
then every properly embedded free-boundary minimal hypersurface $\Sigma$
satisfying either that $\Sigma$ is orientable or that $\pi_1(M)$ is finite
obeys
$\sigma_1(\Sigma)\ge \frac{k}{2}.$
In the Euclidean unit ball they conjectured the sharper
identity $\sigma_1=1$~\cite[Conjecture~3.3]{FraserLi2014}. The following does
not contradict that Euclidean conjecture, but rules out its direct curved-ambient
analogue.

\begin{proposition}\label{prop:fraser-li-disk}
For every sufficiently small $\varepsilon>0$, the disk
$\Sigma=\{y_1=0\}\cap B^3$ is a properly embedded totally geodesic
free-boundary minimal disk in $(B^3,\widehat g_\varepsilon)$ and
\[
 \sigma_1(\Sigma,\widehat g_\varepsilon|_\Sigma)
 \le 1-\frac{93739}{35481600}\,\varepsilon+O(\varepsilon^2)
 <1.
\]
\end{proposition}

Since the main theorem gives $c(\varepsilon)>1$ for all sufficiently small $\varepsilon>0$, the proposition yields
$$
\sigma_1(\Sigma,\widehat g_\varepsilon|_\Sigma)<1<c(\varepsilon).
$$
Thus the lower bound $\sigma_1\ge k/2$ of Fraser--Li remains untouched, but the formally stronger curved-ambient estimate $\sigma_1\ge k$ is false, even for a totally geodesic free-boundary disk in a three-ball arbitrarily close to the Euclidean ball. This does not contradict the Fraser--Li conjecture, which concerns free-boundary minimal surfaces in the Euclidean unit ball.

Since $\Sigma$ is totally geodesic, its Gaussian curvature is the ambient sectional curvature of $T\Sigma$. At $(0,0,1)$ the computation in Remark \ref{rem:section} below gives
\[
K'_\Sigma=-\frac{17}{25}.
\]
Thus the induced disk develops negative Gaussian curvature for
$\varepsilon>0$, as it must: otherwise Escobar's two-dimensional
theorem~\cite[Theorem~1]{Escobar1997}, together with
$k_{\partial\Sigma}\ge c(\varepsilon)$,
would force
$\sigma_1(\Sigma)\ge c(\varepsilon)$,
contradicting the proposition.

\begin{remark}
  The same free-boundary phenomenon is also present in the three-dimensional conformal construction of Sun, Wang, and Wang~\cite{SunWangWang2026}.
Indeed, their symmetry fixes an equatorial disk which is totally geodesic and meets the boundary orthogonally, and the restriction of their deformation to
this disk likewise yields
$\sigma_1(\Sigma)<1<c$
for sufficiently small positive parameter.
\end{remark}

\subsection{Open problem}
The failure of the conjectured constant $1$ leads naturally to the following optimization problem.
For each $n\ge3$, define
\[
 \beta_n
 :=
 \inf\left\{
 \sigma_1(M,g):
 \Ric_g\ge0,\quad
 \II_{\partial M}\ge g_{\partial M}
 \right\},
\]
where the infimum is over smooth compact connected $n$-manifolds with nonempty
boundary.  This normalization entails no loss of generality, since Steklov
eigenvalues and boundary principal curvatures have the same homothetic
scaling.  Escobar's theorem gives
$\beta_n\ge\frac12.$
Our construction shows independently that $\beta_3<1$, while the conformal
counterexamples of Sun, Wang, and Wang~\cite{SunWangWang2026} give
$\beta_n<1$ in every dimension $n\ge3$, and in particular for all $n\ge4$.
Thus
$ \frac12\le\beta_n<1.$
Determining $\beta_n$ is the natural optimization problem left by Escobar's
conjecture.  It is tempting to ask whether $\beta_n=\frac12$ in every
dimension.  Since all presently known counterexamples occur already on
$B^n$, even the corresponding infimum restricted to metrics on the ball is
an intriguing open problem.

\subsection*{Acknowledgments}

We are grateful to Iosif Polterovich for valuable discussions and for his comments on an earlier version of this manuscript.

\section{Variation formulas and homothetic normalization}
Consider a smooth variation of the Euclidean metric
$$g_\varepsilon=g_E+\varepsilon h+O(\varepsilon^2).$$

The linearization of the Ricci tensor is standard; see, for example,
\cite[Theorem~1.174(d)]{Besse1987}.  At the Euclidean metric it reduces to
\begin{equation}\label{eq:linric}
 (D\Ric_{g_E}[h])_{ij}
 =
 \frac12\bigl(
 \partial^p\partial_i h_{pj}
 +\partial^p\partial_j h_{pi}
 -\Delta h_{ij}
 -\partial_i\partial_j\tr h
 \bigr).
\end{equation}

  The first variation of the second fundamental form under a variation of the
ambient metric is also standard; see, for example, \cite[Eq.~(3.4)]{Lott2012}.
With our outward-normal convention it reads
\[
 \dot{\II}_h(X,Y)
 =
 \frac12\bigl(
 (\nabla_\nu h)(X,Y)
 -(\nabla_Xh)(\nu,Y)
 -(\nabla_Yh)(\nu,X)
 +h(\nu,\nu)\II(X,Y)
 \bigr).
\]
Since $S=g_{\partial B^3}^{-1}\II$, differentiation gives
\[
 \langle\dot S_hX,Y\rangle
 =\dot{\II}_h(X,Y)-h(SX,Y).
\]
As $S=I$ and $\II=g_{S^2}$ for the Euclidean unit sphere, we obtain
\begin{equation}\label{eq:shapevar}
 \langle \dot S_hX,Y\rangle
 =\frac12\bigl(
 (\nabla_\nu h)(X,Y)
 -(\nabla_Xh)(\nu,Y)
 -(\nabla_Yh)(\nu,X)
 \bigr)
 -h(X,Y)+\frac12h(\nu,\nu)\langle X,Y\rangle .
\end{equation}

For a fixed smooth function $f$, write
\[
 \cR_\varepsilon[f]
 =\frac{\displaystyle\int_{B^3}|df|_{g_\varepsilon}^2\,dV_{g_\varepsilon}}
 {\displaystyle\int_{S^2}f^2\,dA_{g_\varepsilon}}.
\]
We write
\[
\cR'_0[f;h]
:=
\left.\frac{d}{d\varepsilon}\right|_{\varepsilon=0}
\cR_{\varepsilon}[f]
\]
for the first variation of the Rayleigh quotient of the fixed function $f$
in the metric direction $h$.
If $N_\eps,D_\eps$ are the numerator and denominator, then
\begin{align}
 N'_0&=\int_{B^3}\left(\frac12\tr h\,|df|^2-h(\nabla f,\nabla f)\right)dV,
 \label{eq:Nvar}\\
 D'_0&=\frac12\int_{S^2}f^2\tr_{TS^2}h\,dA,
 \label{eq:Dvar}\\
 \cR'_0[f;h]&=\frac{N'_0D_0-N_0D'_0}{D_0^2}.
 \label{eq:Rvar}
\end{align}

\begin{lemma}\label{lem:hom}
For $h=2ag_E$, $a\in\RR$,
\[
 D\Ric_{g_E}[h]=0,
 \qquad
 \dot S_h=-aI,
 \qquad
 \cR'_0[s;h]=-a.
 \]
Consequently the sign of $\sigma_1-c$ is unchanged by this normalization.
\end{lemma}

\begin{proof}
A constant rescaling does not change the Levi-Civita connection or the Ricci tensor as a covariant two-tensor.  Under $g\mapsto\lambda g$, the unit normal scales by $\lambda^{-1/2}$ and $\II$ by $\lambda^{1/2}$, whereas the induced boundary metric scales by $\lambda$.  Hence the shape operator $S=g_{\partial B^3}^{-1}\II$ scales by $\lambda^{-1/2}$; every Steklov eigenvalue (and every fixed-function Rayleigh quotient) scales by the same factor.  Differentiating at $\lambda=1$ gives the three first-variation statements.

For the final assertion, observe the exact factorization
\[
 g_E+\varepsilon(k+2ag_E)
 =
 (1+2a\varepsilon)
 \left(
 g_E+\frac{\varepsilon}{1+2a\varepsilon}k
 \right).
\]
Thus, setting
$\lambda=1+2a\varepsilon,\delta=\frac{\varepsilon}{\lambda},$
we have
\[
 g_E+\varepsilon(k+2ag_E)=\lambda(g_E+\delta k).
\]
Since both $\sigma_1$ and $c$ scale by $\lambda^{-1/2}$ under
$g\mapsto\lambda g$, the sign of $\sigma_1-c$ is unchanged.
\end{proof}

\section{The explicit perturbation and positive Ricci curvature}\label{sec:explicit}

Set
\[
 u=s^2,
 \qquad
 t=y_1^2+y_2^2.
 \]
 Write
\[
 y_1=\rho\cos\theta,\qquad y_2=\rho\sin\theta,
 \qquad \rho=\sqrt{t}.
\]
Away from the axis $\rho=0$, let $\partial_\rho$ denote the Euclidean unit
radial vector in the $y$-plane and let
\[
 e_\theta=\frac1{\rho}\partial_\theta
\]
denote the Euclidean unit angular vector.
For $a,b\in\{1,2\}$, define the symmetric $2$-tensor
$k\in\Gamma\!\left(\operatorname{Sym}^2 T^*\mathbb R^3\right)$
by
\begin{equation}\label{eq:kansatz}
 k_{ss}=A,\qquad
 k_{sa}=sB\,y_a,\qquad
 k_{ab}=C\delta_{ab}+D y_ay_b.
\end{equation}
where
\begin{align}
 A&=\frac{97}{200}-t+\frac7{200}t^2-\frac{27}{40}u+ut+u^2,
 \label{eq:A}\\
 B&=-1-\frac9{50}t-\frac{11}{200}t^2-u+\frac{18}{25}ut+\frac1{100}u^2,
 \label{eq:B}\\
 C&=-1-t-\frac{53}{200}t^2-\frac1{200}u-ut-\frac{133}{200}u^2,
 \label{eq:C}\\
 D&=1-t-\frac1{20}t^2+u-\frac{17}{50}ut+\frac{37}{40}u^2.
 \label{eq:D}
\end{align}
The particular coefficients are chosen so that the three first-order requirements described above hold simultaneously.
Since $A,B,C,D$ are polynomials in $u$ and $t$, the tensor $k$ is a
Cartesian polynomial tensor of degree six.
Set
\begin{equation}\label{eq:rdef}
 r:=-\frac{53021}{138600},
 \qquad
 \widehat k:=k+2r g_E,
 \qquad
 \widehat g_\varepsilon:=g_E+\varepsilon\widehat k.
\end{equation}
The identity
$\mathcal R'_0[s;k]=r$
will be verified in Section~\ref{sec:rayleigh}. 
By Lemma~\ref{lem:hom}, $D\Ric_{g_E}[\widehat k]=D\Ric_{g_E}[k]$ and $\cR'_0[s;\widehat k]=0$.

Let
\[
 R=D\Ric_{g_E}[k]=D\Ric_{g_E}[\widehat k].
 \]
In the orthonormal frame $(\partial_s,\partial_\rho,e_\theta)$, direct differentiation in \eqref{eq:linric} gives
\begin{align}
 R_{ss}&=\frac{t^2+309tu+10u^2-2u+1}{100},\label{eq:Rss}\\
 R_{s\rho}&=s\rho L,
 \qquad
 L=\frac{-34t+185u+300}{100},\label{eq:L}\\
 R_{\rho\rho}&=\frac{-17t^2+50tu-126t+195u^2+598u+601}{200},\label{eq:Rrr}\\
 R_{\theta\theta}&=P
 =\frac{-41t^2+296tu+174t+195u^2+598u+601}{200}.
 \label{eq:P}
\end{align}
Although $(\partial_\rho,e_\theta)$ is undefined at $t=0$, the Cartesian tensor $R$ is smooth there, and the displayed expressions extend continuously to the axis, where
\[
R_{s\rho}=0,
\qquad
R_{\rho\rho}=R_{\theta\theta}.
\]
In the orthonormal frame $(\partial_s,\partial_\rho,e_\theta)$, $R$ consists of the block
\[
 \begin{pmatrix}R_{ss}&s\rho L\\ s\rho L&R_{\rho\rho}\end{pmatrix},
 \qquad (s\rho)^2=ut,
\]
and the entry $P$.

To prove
\[
D\Ric_{g_E}[k]>0
\]
throughout $\overline{B^3}$, it is enough to show that the above
$2\times2$ meridional block is positive definite and that the angular
eigenvalue $P$ is positive.  By Sylvester's criterion, this reduces to
\[
P>0,\qquad
R_{ss}>0,\qquad
\Delta:=R_{ss}R_{\rho\rho}-utL^2>0
\]
on
\[
\mathcal T=\{(u,t):u\ge0,\ t\ge0,\ u+t\le1\}.
\]

The first two inequalities are immediate.  Since $174t\ge41t^2$ on $0\le t\le1$,
\[
 200P\ge601+174t-41t^2\ge601>0.
\]
Also
\[
 100R_{ss}\ge10u^2-2u+1
 =10\left(u-\frac1{10}\right)^2+\frac9{10}>0.
\]

For the determinant, put $\lambda=1-u-t\ge0.$
We rewrite $20000\Delta$ as a homogeneous quartic in the nonnegative
barycentric variables $u,t,\lambda$, with the aim of expressing it as a
sum of manifestly nonnegative terms.  Namely, every monomial $u^pt^q$
of degree $d<4$ is replaced by
\[
u^pt^q(u+t+\lambda)^{4-d}.
\]
Since $u+t+\lambda=1$ on $\mathcal T$, this does not change the value of
the polynomial there.  Grouping the resulting quartic by powers of
$\lambda$ gives the exact identity
\begin{align}
 20000\Delta={}&
 u^2\bigl(12546u^2-24188ut+23586t^2\bigr)
 +3458ut^3+916t^4\notag\\
 &+\lambda u\bigl(16200u^2-13328ut+20332t^2\bigr)
 +3068\lambda t^3\notag\\
 &+\lambda^2\bigl(6803u^2+11033ut+3812t^2\bigr)\notag\\
 &+\lambda^3(1800u+2278t)+601\lambda^4.
 \label{eq:deltadecomp}
\end{align}
After substituting $\lambda=1-u-t$, this is simply a polynomial identity in $u,t$, so it may be checked by expansion.  The two binary quadratic forms in the first two lines are positive definite, since their leading coefficients are positive and
\[
 4\cdot12546\cdot23586-24188^2=598580480>0,
 \qquad
 4\cdot16200\cdot20332-13328^2=1139878016>0.
\]
All remaining displayed terms have nonnegative coefficients.  The right-hand side cannot vanish: if $\lambda>0$, the last term is positive; if $\lambda=0$, then $u+t=1$, and either $u>0$, when the first summand is positive, or $u=0,t=1$, when $916t^4>0$.  Thus $\Delta>0$ throughout $\mathcal T$.
Together with the positivity of
$P$ and $R_{ss}$, Sylvester's criterion therefore gives
\[
 D\Ric_{g_E}[\widehat k]>0
 \qquad\text{on }\overline{B^3}.
\]
Since $\overline{B^3}$ is compact, there exists $\delta>0$ such that
$ D\Ric_{g_E}[\widehat k]\ge \delta g_E$
throughout the closed ball.  The Ricci tensor depends smoothly on the metric
and its first two derivatives, and hence
\[
 \Ric_{\widehat g_\varepsilon}
 =
 \varepsilon D\Ric_{g_E}[\widehat k]+O(\varepsilon^2)
\]
uniformly on $\overline{B^3}$.  It follows that
$ \Ric_{\widehat g_\varepsilon}>0$
for all sufficiently small $\varepsilon>0$.

\section{Strict first-order boundary convexity}
On $S^2\cap\{\rho>0\}$, write $\rho^2=1-u$.  Since
$ \nu=s\,\partial_s+\rho\,\partial_\rho, s^2+\rho^2=1$,
the vectors
$ w=\rho\,\partial_s-s\,\partial_\rho, e_\theta$
form a Euclidean orthonormal tangent frame.  The formulas below extend to
the poles $\rho=0$ by continuity.

We first justify that $\dot S_k$ is diagonal in this frame.  Fix a boundary point with $\rho>0$ and, after an $O(2)$ rotation, suppose it has $y_2=0$, $y_1>0$.  Reflection in $y_2$ preserves $g_E$ and $k$, hence preserves $\dot S_k$; it fixes the point and $w$ and sends $e_\theta$ to $-e_\theta$.  Therefore
\[
 \dot S_k(w,e_\theta)=-\dot S_k(w,e_\theta)=0.
\]
The poles follow by continuity.

Substitution of \eqref{eq:A}--\eqref{eq:D} into the shape-operator variation formula \eqref{eq:shapevar} gives the two eigenvalues
\begin{align}
 a(u)&:=\dot S_k(w,w)
 =\frac{4941u^4-5205u^3+1651u^2-146u-149}{400},
 \label{eq:a}\\
 b(u)&:=\dot S_k(e_\theta,e_\theta)
 =\frac{549u^4+341u^3+429u^2-78u-149}{400}.
 \label{eq:b}
\end{align}
As a check, $a(1)=b(1)=273/100$, as required by isotropy at the poles.

By \eqref{eq:rdef} and Lemma~\ref{lem:hom}, the eigenvalues of $\dot S_{\widehat k}$ are
\begin{equation}\label{eq:abhat}
 \widehat a(u)=a(u)-r=\frac{\Phi_a(u)}{277200},
 \qquad
 \widehat b(u)=b(u)-r=\frac{\Phi_b(u)}{277200},
\end{equation}
where
\begin{align}
 \Phi_a(u)&=3424113u^4-3607065u^3+1144143u^2-101178u+2785,
 \label{eq:Phia}\\
 \Phi_b(u)&=380457u^4+236313u^3+297297u^2-54054u+2785.
 \label{eq:Phib}
\end{align}
For $u\ge0$, the first two terms of $\Phi_b$ are nonnegative, while the remaining quadratic is positive because
\[
 54054^2-4\cdot297297\cdot2785=-390053664<0.
\]
Thus $\Phi_b(u)>0$ on $[0,1]$.
For $\Phi_a$ we use the exact completion-of-squares identity
\[
\begin{aligned}
\Phi_a(u)
={}&3424113
\left(
u^2-\frac{1735}{3294}u+\frac1{36}
\right)^2\\
&+\frac{967813}{244}
\left(
u-\frac{3111}{25138}
\right)^2
+\frac{8264203}{100552}.
\end{aligned}
\]
Hence $\Phi_a(u)>0$ for every $u\in\mathbb R$.

Therefore
\begin{equation}\label{eq:Sdotpos}
 \dot S_{\widehat k}>0
 \qquad\text{everywhere on }S^2.
\end{equation}

Set $\eta=\min_{0\le u\le1}\{\widehat a(u),\widehat b(u)\}>0$.  The expansion $S_{\widehat g_\varepsilon}=I+\varepsilon\dot S_{\widehat k}+O(\varepsilon^2)$ is uniform in operator norm on the compact boundary.
Therefore, uniformly for $x\in S^2$,
\[
\lambda_{\min}\bigl(S_{\widehat g_\varepsilon}(x)\bigr)
=
1+\varepsilon\lambda_{\min}
\bigl(\dot S_{\widehat k}(x)\bigr)
+O(\varepsilon^2).
\]
Since the minimum defining $\eta$ is attained, it follows that
\[
c(\varepsilon)=1+\eta\varepsilon+O(\varepsilon^2)>1
\]
for all sufficiently small positive $\varepsilon$.

\section{A stationary coordinate Rayleigh quotient}\label{sec:rayleigh}

The reflection $\iota(s,y)=(-s,y)$ is an isometry of $g_E+\varepsilon k$ and of $\widehat g_\varepsilon$.  Hence the boundary area measure is $\iota$-invariant, while $s\circ\iota=-s$, and therefore
\begin{equation}\label{eq:mean0}
 \int_{S^2}s\,dA_{\widehat g_\varepsilon}=0
\end{equation}
for every sufficiently small $\varepsilon$.  Thus $s$ is an admissible trial function in the variational characterization of $\sigma_1$ for the perturbed metric; no harmonicity of $s$ for $\widehat g_\varepsilon$ is required.

The elementary moments used below are
\begin{equation}\label{eq:moments}
 \int_{B^3}u^pt^q\,dV
 =\frac{\pi}{q+1}\,
 \mathrm B\!\left(p+\frac12,q+2\right),
 \qquad
 \int_{S^2}u^p\,dA=\frac{4\pi}{2p+1},
\end{equation}
where
\[
 \mathrm B(a,b)=\int_0^1 x^{a-1}(1-x)^{b-1}\,dx
\]
is the Euler beta function.
For the unnormalized tensor $k$, the two integrands in \eqref{eq:Nvar}--\eqref{eq:Dvar} for $f=s$ are
\begin{equation}\label{eq:Nintegrand}
 \frac12\tr k-k_{ss}
 =\frac{-10t^3-68t^2u-313t^2+185tu^2-400tu-466u^2+133u-497}{400},
\end{equation}
and, on $S^2$,
\begin{equation}\label{eq:Dintegrand}
 \frac12s^2\tr_{TS^2}k
 =\frac{-549u^5+208u^4-261u^3+483u^2-549u}{400}.
\end{equation}
Using \eqref{eq:moments},
\begin{equation}\label{eq:NDprime}
 N'_0=-\frac{47602\pi}{23625},
 \qquad
 D'_0=-\frac{260713\pi}{173250}.
\end{equation}
At $\varepsilon=0$, the numerator and denominator of $\cR_\varepsilon[s]$ are both $4\pi/3$.

Thus
\begin{equation}\label{eq:rcomputed}
 \cR'_0[s;k]
 =\frac{N'_0-D'_0}{4\pi/3}
 =-\frac{53021}{138600}=r.
\end{equation}
Since $\widehat k=k+2r g_E$, Lemma~\ref{lem:hom} gives
\begin{equation}\label{eq:Rstationary}
 \cR'_0[s;\widehat k]=0,
 \qquad
 \cR_\varepsilon[s]=1+O(\varepsilon^2).
\end{equation}

\begin{proof}[Proof of Theorem~\ref{thm:main}]
Since $\widehat k$ is smooth on a neighbourhood of $\overline{B^3}$,
the metric
$\widehat g_\varepsilon=g_E+\varepsilon\widehat k$
is Riemannian for all sufficiently small $|\varepsilon|$.  The Ricci
calculation above showed that
$D\Ric_{g_E}[\widehat k]>0$
uniformly on $\overline{B^3}$; hence
\[
\Ric_{\widehat g_\varepsilon}
=
\varepsilon D\Ric_{g_E}[\widehat k]+O(\varepsilon^2)>0
\]
for all sufficiently small $\varepsilon>0$.  The boundary-shape
calculation and the Rayleigh-quotient calculation give, respectively,
\[
c(\varepsilon)=1+\eta\varepsilon+O(\varepsilon^2),
\qquad
\cR_\varepsilon[s]=1+O(\varepsilon^2),
\qquad
\eta>0.
\]
By \eqref{eq:mean0}, $s$ is admissible for the first nonzero Steklov Rayleigh quotient, so
\[
 \sigma_1(B^3,\widehat g_\varepsilon)
 \le\cR_\varepsilon[s]
 <c(\varepsilon)
\]
for every sufficiently small positive $\varepsilon$.  In particular the boundary is strictly convex, indeed with all principal curvatures $>1$.
\end{proof}

\begin{remark}\label{rem:section}
The example necessarily leaves the class $\Sec\ge0$.  In dimension three, if $(e_1,e_2,e_3)$ is an orthonormal frame, then
\[
 \Ric(e_1,e_1)=K(e_1,e_2)+K(e_1,e_3),
\]
and the two analogous identities obtained by cycling the indices give
\begin{equation}\label{eq:KfromRic}
 K(e_1,e_2)
 =\frac12\bigl(\Ric(e_1,e_1)+\Ric(e_2,e_2)-\Ric(e_3,e_3)\bigr).
\end{equation}
Since the Euclidean curvature tensor vanishes, the same relation holds for the first variations when the frame is fixed at $\varepsilon=0$.  At the boundary equator $(u,t)=(0,1)$, taking $(e_1,e_2,e_3)=(\partial_s,\partial_\rho,e_\theta)$ and using \eqref{eq:Rss}--\eqref{eq:P} gives
\[
 \dot K(\partial_s,\partial_\rho)
 =\frac12\bigl(R_{ss}+R_{\rho\rho}-P\bigr)
 =\frac12\left(\frac2{100}+\frac{458}{200}-\frac{734}{200}\right)
 =-\frac{17}{25}<0.
\]
Thus the construction lies precisely in the gap between nonnegative Ricci curvature and nonnegative sectional curvature, consistently with the theorem of Xia--Xiong.  The homothetic normalization makes the complementary spectral mechanism equally explicit: it leaves the Ricci variation unchanged, makes the coordinate Rayleigh quotient stationary, and makes every boundary principal curvature increase strictly to first order.
\end{remark}

\begin{proof}[Proof of Proposition \ref{prop:fraser-li-disk}]
Writing $y=(y_1,y_2)$, the tensor has the invariant form
\[
 k=A\,ds^2+2sB\,ds\,(y\!\cdot\!dy)+C\,|dy|^2+D\,(y\!\cdot\!dy)^2,
\]
with $A,B,C,D$ depending only on $s^2$ and $|y|^2$. Hence
$(s,y_1,y_2)\mapsto(s,-y_1,y_2)$ is an exact isometry of
$\widehat g_\varepsilon$. Its fixed set $\Sigma$ is totally geodesic; since
the reflection preserves $\partial B^3$, the ambient outward normal along
$\partial\Sigma$ is fixed and therefore tangent to $\Sigma$. Thus $\Sigma$
meets $\partial B^3$ orthogonally.

On $\Sigma$, put $u=s^2$, $t=y_2^2$. The induced trace of $k$ is
$T_\Sigma=A+C+tD$. For the Rayleigh quotient of $s$, the Euclidean numerator
and denominator are both $\pi$, while
\[
 N'_\Sigma(0)=\int_{B^2}\left(\frac12T_\Sigma-A\right)dV
 =-\frac{56623\pi}{76800},\qquad
 D'_\Sigma(0)=\frac12\int_{S^1}s^2\,\operatorname{tr}_{TS^1}k\,dA
 =-\frac{18027\pi}{51200}.
\]
Thus $\cR'_{0,\Sigma}[s;k]=-11833/30720$. On a surface, a homothety
$g\mapsto\lambda g$ multiplies the Steklov Rayleigh quotient by
$\lambda^{-1/2}$; since $\widehat k=k+2rg_E$ with
$r=-53021/138600$,
\[
 \cR'_{0,\Sigma}[s;\widehat k]
 =-\frac{11833}{30720}-r=-\frac{93739}{35481600}<0.
\]
The reflection $s\mapsto-s$ is also an exact isometry, so $s$ has zero
boundary mean for the full perturbed disk metric. The variational principle
and $c(\varepsilon)>1$ now give the claim.
\end{proof}

\section*{AI usage disclosure}

The authors made extensive use of ChatGPT (OpenAI, GPT-5.6) for mathematical
discussions and editorial assistance. In particular, it was used to explore
and test several families of Riemannian metrics on balls in dimensions three
and four before identifying the class of metrics that led to the proof of
Theorem~\ref{thm:main}. All mathematical proofs, computations, and references
appearing in the final manuscript were subsequently checked, revised, and
written in their final form by the authors.

\bibliographystyle{plain} 
\bibliography{BiblioEscobar}

\begin{thebibliography}{10}

\bibitem{Besse1987}
Arthur~L. Besse.
\newblock {\em Einstein Manifolds}, volume~10 of {\em Ergebnisse der Mathematik
  und ihrer Grenzgebiete}.
\newblock Springer-Verlag, Berlin, 1987.

\bibitem{ColboisGirouardGordonSher2024}
Bruno Colbois, Alexandre Girouard, Carolyn Gordon, and David Sher.
\newblock Some recent developments on the {Steklov} eigenvalue problem.
\newblock {\em Revista Matem{\'a}tica Complutense}, 37(1):1--161, 2024.

\bibitem{DuncanKumar}
J.~A.~J. Duncan and A.~Kumar.
\newblock The first {Steklov} eigenvalue on manifolds with non-negative {Ricci}
  curvature and convex boundary.
\newblock {\em J. Geom. Anal.}, 35:Art.~95, 2025.

\bibitem{Escobar1997}
Jos\'e~F. Escobar.
\newblock The geometry of the first non-zero {Stekloff} eigenvalue.
\newblock {\em J. Funct. Anal.}, 150:544--556, 1997.

\bibitem{Escobar1999}
Jos\'e~F. Escobar.
\newblock An isoperimetric inequality and the first {Steklov} eigenvalue.
\newblock {\em J. Funct. Anal.}, 165:101--116, 1999.

\bibitem{FraserLi2014}
Ailana Fraser and Martin Man-chun Li.
\newblock Compactness of the space of embedded minimal surfaces with free
  boundary in three-manifolds with nonnegative {Ricci} curvature and convex
  boundary.
\newblock {\em J. Differential Geom.}, 96(2):183--200, 2014.

\bibitem{GirouardPolterovich2017}
Alexandre Girouard and Iosif Polterovich.
\newblock Spectral geometry of the {Steklov} problem.
\newblock {\em Journal of Spectral Theory}, 7(2):321--359, 2017.

\bibitem{LevitinMangoubiPolterovich2023}
Michael Levitin, Dan Mangoubi, and Iosif Polterovich.
\newblock {\em Topics in Spectral Geometry}, volume 237 of {\em Graduate
  Studies in Mathematics}.
\newblock American Mathematical Society, Providence, RI, 2023.

\bibitem{Lott2012}
John Lott.
\newblock Mean curvature flow in a {R}icci flow background.
\newblock {\em Communications in Mathematical Physics}, 313(2):517--533, 2012.

\bibitem{SunWangWang2026}
Jin Sun, Lili Wang, and Tao Wang.
\newblock Counterexamples to {Escobar}'s conjecture, 2026.
\newblock arXiv:2608.23063.

\bibitem{XiaXiong}
Chao Xia and Changwei Xiong.
\newblock {Escobar}'s conjecture on a sharp lower bound for the first nonzero
  {Steklov} eigenvalue.
\newblock {\em Peking Math. J.}, 7:759--778, 2024.

\end{thebibliography}


\end{document}